\documentclass[a4paper,10pt]{article}
\usepackage[english]{babel}
\usepackage{a4wide}
\usepackage{enumitem}
\usepackage{multicol}
\usepackage[utf8]{inputenc}
\usepackage{todonotes}

\usepackage{amsmath,amssymb,amsthm,array,textpos,float,dsfont,verbatim,fancyvrb,adjustbox,xfp,parskip}
\usepackage{breqn}

\usepackage{mathtools,makecell,graphicx,wrapfig,caption,tikz,tikz-cd,framed,pgffor}
\usetikzlibrary{arrows,positioning,decorations.pathmorphing,decorations.markings,arrows.meta, angles, calc, quotes, tikzmark,chains,shapes.geometric,bending}
\usepackage[labelformat=simple]{subcaption}

\tikzset{
  symbol/.style={
    draw=none,
    every to/.append style={
      edge node={node [sloped, allow upside down, auto=false]{$#1$}}}
  }
}
\tikzset{edge/.style= {line width=0.75pt} }
\tikzset{commutative diagrams/.cd,
	mysymbol/.style={start anchor=center,end anchor=center,draw=none}
}
\PassOptionsToPackage{hyphens}{url}\usepackage{hyperref}

\newtheoremstyle{break}
{\topsep}{\topsep}%
{\itshape}{}%
{\bfseries}{}%
{\newline}{}%

\theoremstyle{definition}
\newtheorem{definition}{Definition}[section]
\newtheorem{example}[definition]{Example}

\theoremstyle{plain}
\newtheorem{theorem}[definition]{Theorem}
\newtheorem{theoremLetter}{Theorem}

\newtheorem{corollary}[definition]{Corollary}
\newtheorem{lemma}[definition]{Lemma}

\theoremstyle{remark}
\newtheorem{remark}[definition]{Remark}

\newtheorem{notation}[definition]{Notation}

\numberwithin{equation}{section}

\DeclareMathAlphabet\wiskunde{U}{msb}{m}{n}
\newcommand{\N}{{\wiskunde N}}
\newcommand{\Z}{{\wiskunde Z}}

\newcommand{\R}{{\wiskunde R}}

\newcommand{\semi}{{\wiskunde o}}
\DeclareMathAlphabet\gotisch{U}{euf}{m}{n}

\DeclareMathOperator{\Aut}{Aut}
\DeclareMathOperator{\End}{End}
\DeclareMathOperator{\Inn}{Inn}

\DeclareMathOperator{\Id}{Id}
\DeclareMathOperator{\Stab}{Stab}
\DeclareMathOperator{\Spec}{Spec_R}
\DeclareMathOperator{\Gl}{GL}
\DeclareMathOperator{\Cay}{Cay}
\DeclareMathOperator{\Image}{Im}
\DeclareMathOperator{\rank}{rank}
\DeclareMathOperator{\diag}{diag}
\newcommand*{\olvarphi}{\overline\varphi}
\newcommand*{\olFa}{\overline{F^a}}
\newcommand*{\olF}{\overline F}
\newcommand{\norm}[1]{\vert #1 \vert}

\makeatletter
\newcommand{\biggg}{\bBigg@\thr@@}
\newcommand{\Biggg}{\bBigg@{3}}

\makeatother

\title{The twisted conjugacy growth of virtually abelian groups}
\author{Karel Dekimpe\thanks{Supported by Methusalem grant METH/21/03 - long term structural funding of the Flemish Government.}, Maarten Lathouwers\thanks{Researcher funded by FWO PhD-fellowship fundamental research (file number: 1102424N).}\\
	Karel.Dekimpe\@@kuleuven.be, Maarten.Lathouwers\@@kuleuven.be}
\date{\today}

\begin{document}
\maketitle

\begin{abstract}
In this paper, we study the asymptotics of several growth functions related to twisted conjugacy on virtually abelian groups. First, we study the twisted conjugacy growth function, which counts the number of twisted conjugacy classes intersecting the ball of radius $r$ around the identity element. Thereafter we study the function that measures the size of the intersection of a given twisted conjugacy class with the balls around the identity element. Finally, we study the number of induced twisted conjugacy classes in certain finite quotients of the given virtually abelian group. In each of these cases we obtain a polynomial asymptotic behaviour of these growth functions.
\end{abstract}
\section*{Introduction}
For any group $G$ that is finitely generated by $S$, we will consider the nested sets $B_G^S(r)$ (with $r\in \N_{>0}$) of $G$ defined by
    \[  B_G^S(r):=\{s_1\dots s_k \:\vert \: k\in \{0,1,\dots,r\}\text{ and } s_1,\dots,s_k\in S\cup S^{-1}\} \]
where $S^{-1}:=\{s^{-1}\:\vert\: s\in S\}$. The \textit{conjugacy growth function} counts the number of conjugacy classes of $G$ that intersect these sets $B_G^S(r)$. In 1988, Babenko was the first to introduce this function (\cite{babe88}). Over the years, many researchers studied the conjugacy growth function of a wide range of groups, for example solvable groups (\cite{bc10}), linear groups (\cite{bclm13}), graph products (\cite{chm23}), $\dots$. We refer to \cite{gs10} for a more extensive and detailed overview.

Instead of considering regular conjugacy, we consider \textit{twisted conjugacy}. More precisely, if $\varphi\in \End(G)$ we say that $x,y\in G$ are $\varphi$-conjugate if there exists some $z\in G$ such that $x=zy\varphi(z)^{-1}$. The equivalence classes with respect to this relation are called the \textit{$\varphi$-twisted conjugacy classes} of $G$ and are denoted as $[x]_{\varphi}$. Similarly as above, we define the \textit{twisted conjugacy growth function} $f_R^S(r)$ to be the number of $\varphi$-twisted conjugacy classes that intersect the set $B_G^S(r)$. In this paper, we start in Section \ref{sec:tc growth Z^n} by determining the twisted conjugacy growth function of finitely generated free abelian groups.

\begin{theoremLetter}\label{thmL:twisted conjugacy growth Zn}
	Let $S$ be any finite generating set of $\Z^n$ and $\varphi\in \End(\Z^n)$, then
	\[ f_R^S(r)=\Theta(r^{n-\rank(\Id_{\Z^n}-\varphi)}) \text{ as } r\to \infty. \]
\end{theoremLetter}

Next, in Section \ref{sec:tc growth virt ab} we use Theorem \ref{thmL:twisted conjugacy growth Zn} to describe the twisted conjugacy growth function with respect to an endomorphism of any finitely generated virtually abelian group. Ciobanu, Hermiller, Holt and Rees studied in 2016 (\cite[section 5]{chhr16}) the conjugacy growth of these groups.

\begin{theoremLetter}\label{thmL:twisted conjugacy growth virtually abelian}
	Let $S$ be any finite generating set of a finitely generated virtually abelian group $G$ and $N\lhd_{\text{fin}}G$ a torsion-free abelian subgroup of $G$. Let $\varphi\in \End(G)$ with $\varphi(N)\subset N$, then
	\[ f_R^S(r)=\Theta\left(r^{\rank(N)-\min\limits_{gN\in G/N}\rank(\Id_N-\tau_g\circ\varphi\vert_N)}\right)\quad \text{and}\quad \#B_G^S(r)=\Theta\left(r^{\rank(N)}\right)\text{ as } r\to \infty \]
	where $\tau_g:G\to G:h\mapsto ghg^{-1}$ denotes conjugation by $g\in G$.
\end{theoremLetter}

To the authors' knowledge, Theorem \ref{thmL:twisted conjugacy growth virtually abelian} gives the first general explicit description of the twisted conjugacy growth function for a family of groups. However, it was defined earlier and has for example been used by Evetts (\cite{evet23}) to determine the conjugacy growth of higher Heisenberg groups. Moreover, Evetts proved earlier that the \textit{conjugacy growth series} of any finitely generated virtually abelian group is rational (\cite{evet19}). Benson showed this in 1983 for the \textit{standard growth series} (\cite{bens83}) of these groups.

Next, in Section \ref{sec:tc class growth} we focus on the \textit{twisted conjugacy class growth} $\beta_{[g_0]_{\varphi}\subset G}^S$ of one fixed $\varphi$-twisted conjugacy class $[g_0]_{\varphi}$. This function counts the number of elements in $[g_0]_{\varphi}$ that belong to the sets $B_G^S(r)$. For any subset $U\subset G$, the function $\beta_{U\subset G}^S$ is also known as the \textit{subset growth} of $U$ in $G$.

\begin{theoremLetter}\label{thmL:twisted conjugacy class growth virtually abelian}
	Let $S$ be any finite generating set of a finitely generated virtually abelian group $G$ and $N\lhd_{\text{fin}}G$ a torsion-free abelian subgroup of $G$. Let $\varphi\in \End(G)$ with $\varphi(N)\subset N$ and $[g_0]_{\varphi}$ a fixed $\varphi$-twisted conjugacy class of $G$, then
	\[\beta_{[g_0]_{\varphi}\subset G}^S(r)=\Theta(r^k) \text{ as }r\to \infty\quad \text{where}\]
        \[k=\max\{\rank(\Id_N-\tau_g\circ\varphi\vert_N)\:\vert\: gN\in G/N \text{ and } [g_0]_{\varphi}\cap p^{-1}(gN)\neq \emptyset\}\]
    where $p:G\to G/N$ denotes the projection of $G$ onto $G/N$.
\end{theoremLetter}

Theorem \ref{thmL:twisted conjugacy class growth virtually abelian} is based on a recent paper by Dermenjian and Evetts (\cite{de23}) in which they argue that the \textit{conjugacy class growth} of a finitely generated virtually abelian group is polynomial.

In Section \ref{sec:quotient growth}, we study the number of induced twisted conjugacy classes in certain finite quotients of $G$. More precisely, if $\Z^n\subset_{\text{fin}} G$, then we consider the finite quotients $G/(k\Z)^n$. Any endomorphism $\varphi\in \End(G)$ with $\varphi(\Z^n)\subset \Z^n$ induces morphisms $\olvarphi_k\in \End(G/(k\Z)^n)$. The \textit{quotient growth} $f_Q$ is defined as the number of $\olvarphi_k$-twisted conjugacy classes of $G/(k\Z)^n$. It turns out that for finitely generated virtually abelian groups, the quotient growth is asymptotically the same as the twisted conjugacy growth.

\begin{theoremLetter}\label{thmL:quotient growth virtually abelian}
	Let $S$ be any finite generating set of a finitely generated virtually abelian group $G$ and $N\lhd_{\text{fin}}G$ a torsion-free abelian subgroup of $G$. Let $\varphi\in \End(G)$ with $\varphi(N)\subset N$, then
	\[f_Q(k)=\Theta(f_R^S(k))=\Theta\left(k^{\rank(N)-\min\limits_{gN\in G/N}\rank(\Id_N-\tau_g\circ\varphi\vert_N)}\right) \text{ as }k\to \infty.\]
\end{theoremLetter}

    We conclude with Section \ref{sec:examples}, in which we give some concrete examples and final remarks. Moreover, we describe a family of groups for which the degree in Theorem \ref{thmL:twisted conjugacy class growth virtually abelian} becomes cleaner.
\begin{theoremLetter}\label{thmL:examples quotient}
	Let $S$ be any finite generating set of a finitely generated virtually abelian group $G$ and $N\lhd_{\text{fin}}G$ a torsion-free abelian subgroup of $G$. Let $\varphi\in \End(G)$ with $\varphi(N)\subset N$. If $G/N$ is abelian and the induced morphism $\olvarphi\in \End(G/N)$ is the identity on $G/N$, then for any $g_0\in G$ it holds that
	\[\beta_{[g_0]_{\varphi}\subset G}^S(r)=\Theta(r^{\rank(\Id_N-\tau_{g_0}\circ\varphi\vert_N)}) \text{ as }r\to \infty.\]
\end{theoremLetter}

\section{Twisted conjugacy growth}
\subsection{Definitions}
For a group endomorphism $\varphi\in \End(G)$ we say that two group elements $g,h\in G$ are \textit{$\varphi$-twisted conjugate} if there exists some $z\in G$ such that $g=zh\varphi(z)^{-1}$ and we denote it with $g\sim_{\varphi}h$. The relation $\sim_{\varphi}$ is an equivalence relation on $G$ of which the equivalence classes are called the \textit{$\varphi$-twisted conjugacy} classes of $G$. We denote the $\varphi$-twisted conjugacy class of $G$ containing $g\in G$ with $[g]_{\varphi}$. The number of $\varphi$-twisted conjugacy classes of $G$ is called the \textit{Reidemeister number $R(\varphi)$} of $\varphi$. The \textit{Reidemeister spectrum $\Spec(G)$} of $G$ contains all Reidemeister numbers of automorphisms of $G$.

We want to study the spread of the twisted conjugacy classes in the \textit{Cayley graph} of $G$. For completeness, we include the relevant definitions with their notation.

\begin{definition}\label{def:word metric}
	Let $G$ be finitely generated by $S$.
	\begin{itemize}
		\item The \emph{Cayley graph $\Cay(S,G)$} is the graph with vertex set $G$ and edge set
		\[ \{\{g,gs\}\:\vert\: g\in G,\: s\in (S\cup S^{-1})\setminus \{1_G\}\} \]
            where $1_G$ denotes the identity element of $G$.
		\item 	The \emph{word metric $d_S$} on $G$ with respect to $S$ is defined for all $g,h\in G$ by
		\[ d_S(g,h):=\min\{k\in \N\:\vert\: \exists s_1,\dots,s_k\in S\cup S^{-1}:\: g^{-1}h=s_1\dots s_k\}. \]
		\item The \emph{word length $\norm{g}_S$} of $g\in G$ with respect to $S$ equals $d_S(g,e)$.
		\item The \emph{closed ball $B_G^S(r)$} of radius $r\in \N$ is defined by
		\[  B_G^S(r):=\{g\in G\:\vert\: \norm{g}_S\leq r\}=\{s_1\dots s_k \:\vert \: k\in \{0,1,\dots,r\}\text{ and } s_1,\dots,s_k\in S\cup S^{-1}\}. \]
            \item The \emph{word growth $\beta_G^S(r)$} denotes the number of elements in $B_G^S(r)$, i.e. $\beta_G^S(r):=\# B_G^S(r)$.
	\end{itemize}
\end{definition}

The \textit{twisted conjugacy growth function} $f_R^S(r)$ counts the number of different twisted conjugacy classes that intersect the closed ball $B_G^S(r)$. Formally we define it in the following way.

\begin{definition}\label{def:twisted conjugacy growth}
	Let $G$ be finitely generated by $S$ and $\varphi\in \End(G)$.
	\begin{itemize}
		\item The \emph{twisted conjugacy growth function $f_R^S$} of $\varphi$ with respect to $S$ is defined by 
		\[ f_R^S:\N\to \N: r\mapsto \# \{[g]_{\varphi} \:\vert\: g\in B_G^S(r)\}.\]
		\item The \emph{twisted conjugacy growth} of $\varphi$ with respect to $S$ is the following limit (if it exists)
		\[ \lim_{r\to \infty} \frac{\log(f_R^S(r))}{\log(\beta_G^S(r))}.\]
	\end{itemize}
\end{definition}

\begin{remark}
	We extend the functions $f_R^S$ and $B_G^S$ to $\R_{\geq 0}$ by taking the floor of any positive real number, so $f_R^S(r)=f_R^S(\lfloor r\rfloor)$ and $B_G^S(r)=B_G^S(\lfloor r\rfloor)$ for all $r\in \R_{\geq 0}$.
\end{remark}
\subsection{Comparing asymptotic growth}
In this section, we describe two different ways to compare functions asymptotically. We provide some relations between them and argue that they are interchangeable when considering polynomial growth. Next, we apply these results to prove that $B_G^S$ and $f_R^S$ grow independent of the generating set. We start by giving the definitions.
\begin{definition}\label{def:quasi-equivalent}
	Let $f,g:\R_{\geq 0}\to \R_{\geq 0}$ be two increasing functions. We say that $g$ \emph{quasi-dominates} $f$, denoted by $f\prec g$, if there exists $b,c\in \R_{>0}$ such that for all $r\in \R_{\geq 0}$ it holds that
	\[ f(r)\leq cg(cr+b)+b. \]
	The functions $f$ and $g$ are called \emph{quasi-equivalent}, denoted by $f\sim g$, if both $f\prec g$ and $g\prec f$.
\end{definition}

\begin{definition}\label{def:big Theta}
	Let $f,g:\R_{\geq 0}\to \R_{\geq 0}$ be two increasing functions. Then $f(r)=\Theta(g(r))$ as $r\to\infty$, if there are $c_1,c_2>0$ and $n_0\in\N_{>0}$ such that for all $r\in \R_{\geq n_0}$ it holds that
	\[ c_1g(r)\leq f(r)\leq c_2g(r). \]
\end{definition}

It turns out that if two functions can be compared using $\Theta$, then they are quasi-equivalent. The converse also holds if the functions grow polynomially.

\begin{lemma}\label{lem:asymptotics growth independent}
	Let $f,g:\R_{\geq 0}\to \R_{\geq 0}$ be two increasing functions. If $f(r)=\Theta(g(r))$ as $r\to \infty$, then $f\sim g$.\\
    Moreover, if there is some $k\in \N$ such that $f\sim (r\mapsto r^k)$, then $f(r)=\Theta(r^k)$ if $r\to \infty$.
\end{lemma}
\begin{proof}
	Suppose that $f(r)=\Theta(g(r))$ as $r\to \infty$. Thus there exist $c_1,c_2\in \R_{>0}$ and $n_0\in \N_{>0}$ such that for all $r\geq n_0$ we have that
	\[ c_1g(r)\leq f(r)\leq c_2g(r). \]
	Define $c:=\max\{1/c_1,c_2,1\}$ and take any $r\in \R_{\geq 0}$ arbitrary. Thus we obtain that
	\begin{align*}
		f(r)&\leq f(r+n_0)\leq c_2g(r+n_0)\leq cg(cr+n_0)+n_0\quad \text{and}\\
		g(r)&\leq g(r+n_0)\leq 1/c_1 f(r+n_0)\leq c f(cr+n_0)+n_0
	\end{align*}
	and thus $f\sim g$.\\
    Suppose $f\sim (r\mapsto r^k)$ (with $k\in \N$). Thus there exist $b,c,b',c'\in \R_{>0}$ such that for all $r\in \R_{\geq 0}$ it holds that
	\[ f(r)\leq c(cr+b)^k+b \quad \text{and} \quad r^k\leq c'f(c'r+b')+b'. \]
	Thus we obtain that for all $r\geq b'$ it holds that
	\[\frac{ \left(\frac{r-b'}{c'}\right)^k-b'}{c'} \leq f(r)\leq c(cr+b)^k+b. \]
	Note that $\frac{ \left(\frac{r-b'}{c'}\right)^k-b'}{c'}=\Theta(r^k)$ and $c(cr+b)^k+b=\Theta(r^k)$ as $r\to \infty$. Thus there exist $n_1,n_2\in \N_{>0}$ and $c_1,c_2\in \R_{>0}$ such that for all $r\geq \max\{n_1,n_2\}$ it holds that
	\[c_1r^k \leq \frac{ \left(\frac{r-b'}{c'}\right)^k-b'}{c'}\quad \text{and}\quad c(cr+b)^k+b \leq c_2 r^k. \]
	Hence, define $n_0:=\lceil \max\{b',n_1,n_2\} \rceil\in \N_{>0}$. Then for any $r\geq n_0$ it holds that
	\[c_1r^k \leq \frac{ \left(\frac{r-b'}{c'}\right)^k-b'}{c'} \leq f(r)\leq c(cr+b)^k+b \leq c_2 r^k. \]
	So we argued that $f(r)=\Theta(r^k)$ as $r\to \infty$.
\end{proof}
\begin{remark}
	One can prove that $(r\mapsto 2^r)\sim (r\mapsto 3^r)$, but that $2^r\neq \Theta(3^r)$ as $r\to \infty$. In particular, quasi-equivalence is more rigid than $\Theta$.
\end{remark}

The following lemma will be used frequently to argue that the generating set has no influence on the growth of some particular functions.

\begin{lemma}\label{lem:V(r) independent of generating set}
	Let $G$ be finitely generated by $S$ and let $R\subset G$ be a subset of $G$. If $L:R\to \mathcal{P}(G)$ is some function assigning to each group element $x\in R$ a subset $L(x)$ of $G$, then
	\[ (r\mapsto\# \{x\in R\:\vert\: L(x)\cap B_G^S(r)\neq \emptyset\})\sim (r\mapsto\# \{x\in R\:\vert\: L(x)\cap B_G^T(r)\neq \emptyset\}) \]
	for any other finite generating set $T$ of $G$.\\
	In particular, suppose that there is some $k\in \N$ such that
	\[ \# \{x\in R\:\vert\: L(x)\cap B_G^S(r)\neq \emptyset\}=\Theta(r^k)\text{ as } r\to \infty, \]
	then this holds for all finite generating sets of $G$.
\end{lemma}
\begin{proof}
	Fix any finite generating set $T$ of $G$. By for example Proposition 5.2.5 in \cite{loh17} there exists some $d\in \R_{>0}$ such that for all $x\in G$ it holds that
	\[ \frac{1}{d}\norm{x}_S\leq \norm{x}_T\leq d\norm{x}_S. \]
	Denote $V^S(r):=\{x\in R\:\vert\: L(x)\cap B_G^S(r)\neq \emptyset\}$ (and similar for $T$). Note that now for all $r\in \R_{\geq 0}$ it holds that
	\[ V^S(r/d)\subset V^T(r) \subset V^S(dr). \]
	Define $c:=\max\{d,1\}$ and note that for all $r\in \R_{\geq 0}$ we obtain that
	\begin{align*}
		\# V^T(r)&\leq \# V^S(dr)\leq c\# V^S(cr)\leq c\# V^S(cr+1)+1\text{ and}\\
		\# V^S(r)&\leq \# V^T(dr)\leq c\# V^T(cr)\leq c\# V^T(cr+1)+1
	\end{align*}
	and thus it follows that $(r\mapsto \# V^T(r))\sim (r\mapsto \# V^S(r))$.\\
	Moreover, if $\# V^S(r)=\Theta(r^k)$ as $r\to \infty$ (for some $k\in \N$), then by Lemma \ref{lem:asymptotics growth independent} and the above, we have that $\# V^T(r)=\Theta(r^k)$ as $r\to \infty$.
\end{proof}

Lemma \ref{lem:V(r) independent of generating set} can be used to argue that the growth of $B_G^S$ and $f_R^S$ is independent of the generating set when considering quasi-equivalence. Moreover, if they grow polynomially, then together with Lemma \ref{lem:asymptotics growth independent} we obtain the same result when considering $\Theta$.

\begin{corollary}\label{cor:B_G and f_R independent of generating set}
	Let $S$ and $T$ be two finite generating sets for $G$. Then
	\[ (r\mapsto \beta_G^S(r))\sim (r\mapsto\beta_G^T(r)) \quad \text{and} \quad (r\mapsto f_R^S(r))\sim (r\mapsto f_R^T(r)). \]
\end{corollary}

\begin{corollary}\label{cor:twisted conjugacy growth independent generating set}
	Let $G$ be finitely generated by $S$. If there exists $k,l\in \N$ with $l>0$ such that $f_R^S(r)=\Theta(r^k)$ and $\beta_G^S(r)=\Theta(r^l)$ as $r\to \infty$, then this is independent of the generating set. Moreover, the twisted conjugacy growth is independent of the generating set and equals $k/l$.
\end{corollary}


\section{Twisted conjugacy growth of \texorpdfstring{$\Z^n$}{Z^n}}\label{sec:tc growth Z^n}
In this section, we determine the twisted conjugacy growth of $\Z^n$. The main argument is given in the following lemma. The twisted conjugacy growth of $\Z^n$ will then be a direct corollary.
\begin{lemma}\label{lem:V_Id^S(r)}
	Define for any finite generating set $S$ of $\Z^n$ and $B\in \Z^{n\times n}$ the set
	\[ V^S(r):=\left\{x+\Image(B)\:\middle\vert\: (x+\Image(B))\cap B_{\Z^n}^S(r)\neq \emptyset\right\}\subset \Z^n/\Image(B).\]
	Then $\# V^S(r)=\Theta(r^{n-\rank(B)})$ as $r\to\infty$ for all finite generating sets $S$ of $\Z^n$.
\end{lemma}

\begin{proof}
    Note that by Lemma \ref{lem:V(r) independent of generating set} the statement holds for any finite generating set if we are able to prove it for one specific finite generating set $S$ of $\Z^n$. We will use this approach in the proof and construct a specific set $S$.
    
    Denote with $\Lambda=\diag(d_1,\dots,d_l,0,\dots,0)$ the Smith normal form of $B$ where $d_1|d_2|\dots|d_l$ with $d_i\in \N_{>0}$ and $P,Q\in \Gl_n(\Z)$ invertible matrices such that $P^{-1}BQ=\Lambda$. Note that $l=\rank(B)$. Denote with $S=\{v_1,\dots,v_n\}$ the columns of $P$ and with $T=\{w_1,\dots,w_n\}$ the columns of $Q$. Hence, it holds for all $i=1,\dots,n$ that
	\[ Bw_i=\begin{cases}
		d_iv_i &\text{if } i\leq l\\
		0 &\text{if } i>l
	\end{cases}. \]
	Since $S$ and $T$ are both $\Z$-bases of $\Z^n$, it follows that $\Image(B)=\{\sum_{i=1}^{l}z_id_iv_i\:\vert\: z_i\in \Z\}$. Note that the size of $V^S(r)$ does not change if we assume that $S=\varepsilon$ (the standard generating set of $\Z^n$) and that $B=\Lambda$. Hence, we assume this and thus
	\[ B_{\Z^n}^S(r)=\left\{(z_1,\dots,z_n)\in \Z^n\:\middle\vert\: \sum_{i=1}^{n}|z_i|\leq r \right\} \quad \text{and} \quad \Image(B)=\{(z_1d_1,\dots,z_ld_l,0,\dots,0)\:\vert\: z_i\in \Z\}. \]
	
    Define the sets $V\subset \Z^n$ and $\tilde{V}(r)\subset B_{\Z^n}^S(r)$ (for all $r\in \R_{\geq 0}$) by
	\[ V:=\left\{(z_1,\dots,z_n)\in \Z^n\:\middle\vert\: \forall i\leq l:\: \left\lfloor -\frac{d_i}{2}\right\rfloor +1\leq z_i \leq \left\lfloor \frac{d_i}{2}\right\rfloor \right\} \quad \text{and} \quad \tilde{V}(r):=V\cap B_{\Z^n}^S(r). \]
	We claim that $\tilde{V}(r)$ contains precisely one representative of each coset of $\Image(B)$ in $\Z^n$ that intersect $B_{\Z^n}^S(r)$. In particular, it consists of so-called \textit{minimal representative elements} of each such coset. For this, define the map $\tilde{\theta}$ (for all $r\in \R_{\geq 0}$) by setting
	\[ \tilde{\theta}:\tilde{V}(r)\to V^S(r):x\mapsto x+\Image(B). \]
	We argue that $\tilde{\theta}$ is a bijection and thus $\# V^S(r)=\# \tilde{V}(r)$ for all $r\in \R_{\geq 0}$. For the injectivity, take $(x_1,\dots,x_n),(y_1,\dots,y_n)\in \tilde{V}(r)$ with $x+\Image(B)=y+\Image(B)$. Hence, there exist $z_i\in \Z$ such that $(x_1-y_1,\dots,x_n-y_n)=(z_1d_1,\dots,z_ld_l,0,\dots,0)$. Note that for all $i\leq l$ it holds that $|x_i-y_i|\leq d_i-1$ and thus (since $d_i|(x_i-y_i)$) it follows that $(x_1,\dots,x_n)=(y_1,\dots,y_n)$. So $\tilde{\theta}$ is injective.\\
    For the surjectivity, fix an arbitrary $(x_1,\dots,x_n)+\Image(B)\in V^S(r)$ with representative $x=(x_1,\dots,x_n)\in B_{\Z^n}^S(r)$. Write for all $i=1,\dots,l$
	\[ x_i=q_id_i+a_i \]
	where $q_i,a_i\in \Z$ and $\left\lfloor -\frac{d_i}{2}\right\rfloor +1\leq a_i \leq \left\lfloor \frac{d_i}{2}\right\rfloor$. Define $y:=(a_1,\dots,a_l,x_{l+1},\dots,x_n)\in V$ and note that $x-y=(q_1d_1,\dots,q_ld_l,0,\dots,0)\in \Image(B)$. So it suffices to argue that $y\in B_{\Z^n}^S(r)$. We claim that $|a_i|\leq|x_i|$ for all $i\leq l$. Indeed, if $q_i=0$ then $x_i=a_i$ and thus clearly $|a_i|\leq|x_i|$. Assume that $q_i \neq 0$. Hence, we have that
	\[ |x_i-a_i|=|q_i|d_i\geq d_i. \]
	By the choice of $a_i$, we have that $|a_i|\leq \left\lfloor \frac{d_i}{2}\right\rfloor$. Also $\left\lfloor \frac{d_i}{2}\right\rfloor \leq |x_i|$, since otherwise we would have
	\[ |x_i-a_i|\leq |x_i|+|a_i|< \left\lfloor \frac{d_i}{2}\right\rfloor+\left\lfloor \frac{d_i}{2}\right\rfloor\leq d_i \]
	which contradicts that $|x_i-a_i|\geq d_i$. Hence, it follows that $|a_i|\leq \left\lfloor \frac{d_i}{2}\right\rfloor\leq |x_i|$ for all $i\leq l$. Since $x\in B_{\Z^n}^S(r)$ we now find
	\[ \norm{y}_S=\sum_{i=1}^{l} |a_i|+\sum_{i=l+1}^{n} |x_i|\leq \sum_{i=1}^n |x_i|\leq r \]
	and thus $y\in V\cap B_{\Z^n}^S(r)=\tilde{V}(r)$. Concluding, we obtain that $\tilde{\theta}(y)=y+\Image(B)=x+\Image(B)$ and thus $\tilde{\theta}$ is surjective.
	
    Since $\tilde{\theta}$ is bijective, we find that $\# V^S(r)=\# \tilde{V}(r)$. Denote with $V'\subset \Z^l$ the set
	\[ V':=\left\{(z_1,\dots,z_l)\in \Z^l\:\middle\vert\: \forall i\leq l:\: \left\lfloor -\frac{d_i}{2}\right\rfloor +1\leq z_i \leq \left\lfloor \frac{d_i}{2}\right\rfloor \right\} \]
	and note that for all $r\geq \sum_{i\leq l} d_i$ it holds that
	\begin{align*}
		\# \tilde{V}(r)&=\sum_{(z_1,\dots,z_l)\in V'} \#\left\{(x_1,\dots,x_n)\in \Z^n\:\middle\vert\: x_i=z_i, \:\forall i\leq l \text{ and } \norm{x}_S\leq r \right\}\\
		&=\sum_{(z_1,\dots,z_l)\in V'} \#\left\{(x_1,\dots,x_n)\in \Z^n\:\middle\vert\: x_i=z_i, \:\forall i\leq l \text{ and } \sum_{i>l} |x_i|\leq r-\sum_{i\leq l}|z_i| \right\}\\
		&=\sum_{(z_1,\dots,z_l)\in V'} \beta_{\Z^{n-l}}^{\varepsilon}\left(r-\sum_{i\leq l}|z_i|\right)
	\end{align*}
	Since $B_{\Z^{n-l}}^{\varepsilon}(r)=\Theta(r^{n-l})$ as $r\to \infty$ (see for example \cite[Example 6.1.2]{loh17}) and $\# V'<\infty$, it follows that
	\[ \# V^S(r)=\# \tilde{V}(r)=\Theta(r^{n-\rank(B)}) \text{ as } r\to \infty. \]
	By Lemma \ref{lem:V(r) independent of generating set} this holds for all finite generating sets $S$ of $\Z^n$.
\end{proof}

Lemma \ref{lem:V_Id^S(r)} allows us to determine the twisted conjugacy growth of $\Z^n$.

\begin{theorem}\label{thm:twisted conjugacy growth Z^n}
	Let $S$ be any finite generating set of $\Z^n$ and $\varphi\in \End(\Z^n)$, then
	\[ f_R^S(r)=\Theta(r^{n-\rank(\mathds{1}_n-\varphi)}) \text{ as } r\to \infty. \]
	In particular, the twisted conjugacy growth of $\varphi$ equals
	\[ 1-\frac{\rank(\mathds{1}_n-\varphi)}{n}. \]
\end{theorem}
\begin{proof}
	Fix any finite generating set $S$ of $\Z^n$ and any $\varphi\in \End(\Z^n)$. Note that for all $x\in \Z^n$ it holds that
	\[ [x]_{\varphi}=x+\Image(\mathds{1}_n-\varphi) \]
	and thus for all $r\in \R_{\geq 0}$ it holds that
	\[ \{[x]_{\varphi}\:\vert\: x\in B_{\Z^n}^S(r)\}=\left\{x+\Image(\mathds{1}_n-\varphi)\:\middle\vert\: (x+\Image(\mathds{1}_n-\varphi))\cap B_{\Z^n}^S(r)\neq \emptyset\right\}.\]
	Hence, by Lemma \ref{lem:V_Id^S(r)} it follows that $f_R^S(r)=\Theta(r^{n-\rank(\mathds{1}_n-\varphi)})$ as $r\to \infty$. Since $B_{\Z^n}^{S}(r)=\Theta(r^n)$ as $r\to \infty$ (which is contained in example \cite[Example 6.1.2]{loh17}), the result follows by Corollary \ref{cor:twisted conjugacy growth independent generating set}.
\end{proof}


\section{Twisted conjugacy growth of virtually abelian groups}\label{sec:tc growth virt ab}
In this section, we determine the twisted conjugacy growth of all finitely generated virtually abelian groups. For this, we need a generalisation of Lemma \ref{lem:V_Id^S(r)}. In the proof, we describe an upper and lower bound of the considered function and use Lemma \ref{lem:V_Id^S(r)} to obtain the asymptotics of these bounds.
\begin{lemma}\label{lem:layer A}
	Let $B\in \Z^{n\times n}$ and $M_1=\mathds{1}_n,M_2,\dots,M_k\in \Gl_n(\Z)$ (with $k,n\in \N_{>0}$) such that $B$ commutes with all matrices $M_i$. Let $c_1=(0,\dots,0),c_2,\dots,c_k\in \Z^n$. Define for any finite generating set $S$ of $\Z^n$ the set
	\[ V^S(r):=\left\{x+\Image(B)\:\middle\vert\: \left( \bigcup_{i=1}^k (M_ix+c_i+\Image(B))\right)\cap B_{\Z^n}^S(r)\neq \emptyset\right\}.\]
	Then $\# V^S(r)=\Theta(r^{n-\rank(B)})$ as $r\to\infty$ for any $S$.
\end{lemma}
\begin{proof}
	We let $S$ be the standard generating set of $\Z^n$ and thus $\norm{(x_1,\dots,x_n)}_S=\sum_{i=1}^n|x_i|$. Note that for $i=1,\dots,k$ the condition $\left(M_ix+c_i+\Image(B)\right)\cap B_{\Z^n}^S(r)\neq \emptyset$ is independent of the representative $x$ of the coset $x+\Image(B)$ since $B$ and $M_i$ commute. Hence, we can fix a set $R\subset \Z^n$ of representatives of $\Z^n/\Image(B)$ and redefine
	\[ V^S(r):=\left\{x\in R\:\middle\vert\: \left( \bigcup_{i=1}^k (M_ix+c_i+\Image(B))\right)\cap B_{\Z^n}^S(r)\neq \emptyset\right\}\subset \Z^n.\]
	Define for any $i=1,\dots,k$ the set
	\[ V_i^S(r):=\left\{x\in R\:\middle\vert\: \left(M_ix+c_i+\Image(B)\right)\cap B_{\Z^n}^S(r)\neq \emptyset\right\}.\]
	Note that $V_1^S(r)\subset V^S(r)$ for any $r\in \R_{\geq 0}$ and thus
	\[ \# V_1^S(r)\leq \# V^S(r). \]
	Define for any matrix $A\in \Z^{n\times n}$ the norm $\norm{A}:=\sum_{i,j}|A_{ij}|$. Note that since $S$ is the standard generating set of $\Z^n$, for any $x\in \Z^n$ it holds that $\norm{Ax}_S\leq \norm{A}\:\norm{x}_S$. Fix any $i=1,\dots,k$. We argue that for all $r\in \R_{\geq 0}$ it holds that
	\[ V_i^S(r)\subset V_1^S(\norm{M_i^{-1}}(r+\norm{c_i}_S)). \]
	For this, take some $x\in V_i^S(r)$ arbitrary. So there is some $y\in \left(M_ix+c_i+\Image(B)\right)\cap B_{\Z^n}^S(r)$. Note that since $B$ and $M_i$ commute, it follows that $M_i^{-1}(y-c_i)\in x+\Image(B)$. Moreover, it holds that
	\[ \norm{M_i^{-1}(y-c_i)}_S\leq \norm{M_i^{-1}}\: \norm{y-c_i}_S\leq \norm{M_i^{-1}}(r+\norm{c_i}_S). \]
	Hence, we obtain that $M_i^{-1}(y-c_i)\in (x+\Image(B))\cap B_{\Z^n}^S(\norm{M_i^{-1}}(r+\norm{c_i}_S))$ and thus indeed $x\in V_1^S(\norm{M_i^{-1}}(r+\norm{c_i}_S))$.\\
	Note that now
	\[ V^S(r)\subset \bigcup_{i=1}^k V_i^S(r)\subset \bigcup_{i=1}^k V_1^S(\norm{M_i^{-1}}(r+\norm{c_i}_S))\]
	and thus in particular we have that
	\[ \# V_1^S(r) \leq \# V^S(r) \leq \sum_{i=1}^k \# V_1^S(\norm{M_i^{-1}}(r+\norm{c_i}_S)). \]
	By Lemma \ref{lem:V_Id^S(r)} we know that $\# V_1^S(r)=\Theta(r^{n-\rank(B)})$ as $r\to \infty$ and thus the above implies that $\# V^S(r)=\Theta(r^{n-\rank(B)})$ as $r\to \infty$. By applying Lemma \ref{lem:V(r) independent of generating set} we now obtain that $\# V^S(r)=\Theta(r^{n-\rank(B)})$ as $r\to \infty$ for all finite generating sets $S$ of $\Z^n$.
\end{proof}

If $G$ is a finitely generated virtually abelian group, then $\Z^n\subset_{\text{fin}} G$. The next lemma will allow us to use the subset $V^S(r)$ of cosets of $\Z^n$ from Lemma \ref{lem:layer A} to describe the twisted conjugacy growth of $G$.

\begin{lemma}\label{lem:inclusion balls}
	Let $G$ be a finitely generated group and $H\subset_{\text{fin}}G$ a finite index subgroup. Then for any $g\in G$ the map $i^g:H\to G:h\mapsto hg$ is a quasi-isometry (with respect to any word metrics on H and G coming from finite generating sets).\\
    In particular, for all finite generating sets $S$ of $G$ and $T$ of $H$ and all $g\in G$ there exist $r_0\in \R_{\geq 0}$, $c_1,c_2\in \R_{>0}$ such that for all $r\geq r_0$ it holds that
	\[ i^g(B_H^T(c_1r))\subset B_G^S(r)\cap i^g(H)\subset i^g(B_H^T(c_2r)). \]
\end{lemma}
\begin{proof}
	The first statement is inspired by Corollary 5.4.5 in \cite{loh17} and follows directly from its proof which is based on the \v{S}varc-Milnor lemma (see for example \cite[Proposition 5.4.1]{loh17}).\\
    The second statement only uses that $i^g$ is a quasi-isometric embedding, i.e. there exists some $c\in \R_{>0}$ such that for all $h,h'\in H$ we have that
	\[ \frac{1}{c}d_T(h,h')-c \leq d_S(i^g(h),i^g(h')) \leq c d_T(h,h')+c. \]
	Define
	\[ c_1:=\frac{1}{2c},\: c_2:=c+1\text{ and } r_0=\max\{2(c+\norm{g}_S),c(c+\norm{g}_S)\}. \]
	Fix any $r\geq r_0$ and any $h\in B_H^T(c_1r)$. Thus we find that
	\[ \norm{i^g(h)}_S \leq d_S(i^g(h),i^g(1_G))+d_S(g,1_G) \leq c \norm{h}_T+c+\norm{g}_S \leq cc_1r+c+\norm{g}_S \leq \frac{r}{2}+\frac{r}{2}=r \]
	and thus indeed $i^g(h)\in B_G^S(r)\cap i^g(H)$.\\
	Take any $i^g(h)\in B_G^S(r)\cap i^g(H)$ and note that
	\[ \norm{h}_T \leq c d_S(i^g(h),i^g(1_G))+c^2 \leq c \norm{i^g(h)}_S + c\norm{g}_S +c^2 \leq cr+r = c_2r \]
	and thus indeed $i^g(h)\in B_H^T(c_2r)$.
\end{proof}

\begin{remark}
Let $G$ be a finitely generated virtually abelian group. So $G$ contains a finitely generated normal abelian subgroup $N_1$ which is of finite index in $G$. Without loss of generality, we may assume that 
$N_1$ is torsion-free (if $l$ is the order of the torsion subgroup of $N_1$, we may replace 
$N_1$ with its finite index characteristic subgroup which is generated by all $n^l$ with $n\in N_1$).
Now let $k=[G:N_1]$ and let $N$ be the subgroup of $G$ which is generated by all elements of the form 
$g^k$. Then $N$ is a fully characteristic subgroup of $G$, which is contained in $N_1$ and so is also free abelian. Note that $N$ is of finite index in $N_1$ (because it contains $\{n_1^k \mid n_1\in N_1\}$ which is clearly a finite index subgroup of $N_1$) hence $N$ is also of finite index in $G$.

Thus any finitely generated group $G$ contains some $N\lhd_{\text{fin}}G$ which is a torsion-free abelian fully characteristic subgroup of $G$. In the rest of the paper, we do not require $N$ to be fully characteristic, but we only need that $N$ is invariant under the fixed morphism $\varphi\in \End(G)$.
\end{remark}

\begin{notation}\label{not:set-up and definition Ea}
    Let $G$ be a finitely generated virtually abelian group and $N\lhd_{\text{fin}}G$ a torsion-free abelian subgroup of $G$. We fix a finite set $A\subset G$, containing $1_G$, of representatives of $G/N$. Denote for any $a\in A$ with $\tau_a\in \Inn(G)$ conjugation with $a$ in $G$ and with $p:G\to G/N$ the projection of $G$ onto $G/N$. If $\varphi\in \End(G)$ with $\varphi(N)\subset N$, then the intersection of any twisted conjugacy class with a coset $aN\in G/N$ (with $a\in A$) is a finite union of cosets of $\Image(\Id_N-\tau_a\circ\varphi\vert_N)$. Before formulating this result in more detail, let us define for any $a\in A$ the finite set $E^a\subset A$ and the map $\varphi^a:A\to G$ by setting
    \begin{equation}\label{eq:definition Ea}
        \begin{split}
            E^a:&=p^{-1}(\Stab_{\olvarphi}(p(a)))\cap A\\
            &=\{b\in A\:\vert \: p(a)=p(ba\varphi(b)^{-1})\}
        \end{split}\quad \text{and} \quad\begin{split}\varphi^a:A\to G:b\mapsto 
            \varphi^a(b)&:=([a,b^{-1}]^{\varphi})^{-1}\\&:=\varphi(b)a^{-1}b^{-1}a
        \end{split}
    \end{equation}
    where $\olvarphi\in \Aut(G/N)$ is the induced automorphism on $G/N$, $\Stab_{\olvarphi}(p(a))$ is the \textit{$\varphi$-stabiliser} of $p(a)$ (i.e. the stabiliser of $p(a)$ under the $\olvarphi$-conjugacy action on $G/N$) and $[x,y]^{\varphi}:=x^{-1}y^{-1}x\varphi(y)$ is the \textit{$\varphi$-twisted commutator} of any $x,y\in G$. Remark that $\varphi^a(E^a)\subset N$ and $1_G\in E^a$.
\end{notation}

\begin{lemma}\label{lem:virt ab finite union}
    With the set-up from Notation \ref{not:set-up and definition Ea}, it holds for any $a\in A$ and for any $x,y\in N$ that
    \[i^a(x)\sim_{\varphi} i^a(y)\quad \Longleftrightarrow\quad x\in \bigcup_{c\in E^a} (\tau_c(y)-(\tau_a\circ\varphi^a)(c)+\Image(\Id_N-\tau_a\circ\varphi\vert_N)).\]
\end{lemma}
\begin{proof}
        Fix some $a\in A$ and recall the definition of $i^a$ in Lemma \ref{lem:inclusion balls} and of $E^a$ and $\varphi^a$ in Equation (\ref{eq:definition Ea}). Note that $\varphi^a(E^a)\subset N$ by definition. It holds for all $x,y\in N$ that
	\begin{align*}
		i^a(x)\sim_{\varphi} i^a(y)&\Longleftrightarrow \exists z\in N,\: c\in A:\: i^a(x)=i^c(z)i^a(y)\varphi(i^c(z))^{-1}\\
		&\Longleftrightarrow \exists z\in N,\: c\in A:\: i^a(x)=i^c(z)i^a(y)\varphi(c)^{-1}\varphi(z)^{-1}
	\end{align*}
        Since $\varphi(N)\subset N$, we find by applying the morphism $p:G\to G/N$ to the last equality that $p(a)=p(ca\varphi(c)^{-1})$ or thus $c\in E^a$. Hence, we obtain that
	\begin{align*}
		i^a(x)\sim_{\varphi} i^a(y)&\Longleftrightarrow \exists z\in N,\: c\in E^a:\: i^a(x)=i^c(z)i^a(y)(\varphi^a(c)a^{-1}ca)^{-1}\varphi(z)^{-1}\\							  
		&\Longleftrightarrow \exists z\in N,\: c\in E^a:\: x=z\tau_c(y)(\tau_a\circ\varphi^a)(c)^{-1}(\tau_a\circ\varphi)(z)^{-1}\\
		&\Longleftrightarrow \exists z\in N,\: c\in E^a:\: x=\tau_c(y)-(\tau_a\circ\varphi^a)(c)+(\Id_N-\tau_a\circ\varphi\vert_N)(z)
	\end{align*}
	where in the last step we used that $\varphi^a(E^a)\subset N$, that $N$ is abelian and we switched to additive notation for convenience. Hence, we obtain for all $x,y\in N$ that
		\[i^a(x)\sim_{\varphi} i^a(y)\quad \Longleftrightarrow\quad x\in \bigcup_{c\in E^a} (\tau_c(y)-(\tau_a\circ\varphi^a)(c)+\Image(\Id_N-\tau_a\circ\varphi\vert_N)).\]
\end{proof}

\begin{remark}\label{rem:virt ab finite union}
        By Lemma \ref{lem:virt ab finite union} it especially follows for all $a\in A$ and $x\in N$ that
        \[x\in \bigcup_{c\in E^a} (\tau_c(x)-(\tau_a\circ\varphi^a)(c)+\Image(\Id_N-\tau_a\circ\varphi\vert_N)).\]
        Moreover, since $1_G\in E^a$ and $(\tau_a\circ\varphi^a)(1_G)=1_G\in N$, we also obtain that $x+\Image(\Id_N-\tau_a\circ\varphi\vert_N)$ is always contained in the above finite union.
\end{remark}

Lemmas \ref{lem:layer A}, \ref{lem:inclusion balls} and \ref{lem:virt ab finite union} give us the necessary tools to determine the twisted conjugacy growth of any finitely generated virtually abelian group.

\begin{theorem}\label{thm:twisted conjugacy growth virtually abelian}
	Let $S$ be any finite generating set of a finitely generated virtually abelian group $G$ and $N\lhd_{\text{fin}}G$ a torsion-free abelian subgroup of $G$. Let $\varphi\in \End(G)$ with $\varphi(N)\subset N$, then
	\[ f_R^S(r)=\Theta\left(r^{\rank(N)-\min\limits_{gN\in G/N}\rank(\Id_N-\tau_g\circ\varphi\vert_N)}\right)\quad \text{and}\quad \#B_G^S(r)=\Theta\left(r^{\rank(N)}\right)\text{ as } r\to \infty \]
	where $\tau_g:G\to G:h\mapsto ghg^{-1}$ denotes conjugation by $g\in G$.\\
	In particular, the twisted conjugacy growth of $\varphi$ equals
	\[ 1-\frac{\min\limits_{gN\in G/N}\rank(\Id_N-\tau_g\circ\varphi\vert_N)}{\rank(N)}. \]
\end{theorem}
\begin{proof}
	Fix a finite set $A\subset G$ of different representatives of the cosets of $N$ in $G$. We can assume without loss of generality that $1_G\in A$. Recall the definition of $i^a$ in Lemma \ref{lem:inclusion balls} and note that $i^a(N)=p^{-1}(aN)$, where $p:G\to G/N$ is the projection onto $G/N$. Define for any $a\in A$ (and any $r\in \R_{\geq 0}$) the sets
	\[ F(r):=\{[g]_{\varphi}\: \vert\: g\in B_G^S(r)\} \quad \text{and} \quad F_G^a(r):=\{[i^a(x)]_{\varphi}\: \vert\: x\in N,\: i^a(x)\in B_G^S(r)\}. \]
	Thus the set $F_G^a(r)$ precisely contains the twisted conjugacy classes of $\varphi$ that intersect the ball $B_G^S(r)$ in an element $g\in G$ with $p(g)=aN$. It easily follows that
	\[ F(r)= \bigcup_{a\in A} F_G^a(r) \]
	and thus
	\begin{equation}\label{eq:proof inequalities different layers}
		\max_{a\in A} \# F_G^a(r)\leq \#F(r)=f_R^S(r) \leq \sum_{a\in A} \# F_G^a(r).
	\end{equation}
	We now argue that $\# F_G^a(r)=\Theta(r^{\rank(N)-\rank(\Id_N-\tau_a\circ\varphi\vert_N)})$ as $r\to \infty$ for all $a\in A$. Recall from Lemma \ref{lem:virt ab finite union} that for all $x,y\in N$ it holds that
        \[i^a(x)\sim_{\varphi} i^a(y)\quad \Longleftrightarrow\quad x\in \bigcup_{c\in E^a} (\tau_c(y)-(\tau_a\circ\varphi^a)(c)+\Image(\Id_N-\tau_a\circ\varphi\vert_N))\]
        where $E^a$ and $\varphi^a$ are defined in Equation (\ref{eq:definition Ea}). Since for all $c\in E^a$ it holds that $\tau_c\in \Aut(N)$, the automorphism corresponds with multiplication on the left with a unique matrix $M_c\in \Gl_{\rank(N)}(\Z)$.
	We check the conditions of Lemma \ref{lem:layer A} with $B:=\Id_N-\tau_a\circ\varphi\vert_N\in \End(N)$. Note that $1_G\in E^a$, that $M_{1_G}=\mathds{1}_{\rank(N)}$ and that $(\tau_a\circ\varphi^a)(1_G)=1_G\in N$. Moreover, note that for any $c\in E^a$ the map $\tau_c\vert_N$ commutes with $\tau_a\circ\varphi\vert_N$. Indeed, fix any $x\in N$. Since $\varphi(c)=\varphi^a(c)a^{-1}ca$ with $\varphi^a(c)\in N$ it holds that
	\begin{align*}
			(\tau_a\circ \varphi\vert_N\circ\tau_c)(x)&=a\varphi(cxc^{-1})a^{-1}=a\varphi^a(c)a^{-1}ca\varphi(x)a^{-1}c^{-1}a\varphi^a(c)^{-1}a^{-1}\\
			&=(\tau_a\circ\varphi^a)(c)(\tau_c\circ \tau_a\circ \varphi\vert_N)(x)(\tau_a\circ\varphi^a)(c)^{-1}=(\tau_c\circ \tau_a\circ \varphi\vert_N)(x)
		\end{align*}
	where in the last step we used that the three elements belong to $N$, which is abelian. Hence, for any $c\in E^a$ the map $\tau_c\vert_N$ commutes with $\Id_N-\tau_a\circ\varphi\vert_N$ or equivalently the corresponding matrices commute. Fix a set $R\subset N$ of representatives of different cosets of $\Id_N-\tau_a\circ\varphi\vert_N$ in $N$ and fix any finite generating set $T$ of $N\cong \Z^{\rank(N)}$. Lemma \ref{lem:layer A} shows that if we define
	\begin{align*}
		V^a(r):&= \left\{x\in R\:\middle\vert\: \left( \bigcup_{c\in E^a} (\tau_c(x)-(\tau_a\circ\varphi^a)(c)+\Image(\Id_N-\tau_a\circ\varphi\vert_N)) \right)\cap B_{N}^{T}(r)\neq \emptyset\right\} \text{ then }\\
		\#V^a(r)&=\Theta(r^{\rank(N)-\rank(\Id_N-\tau_a\circ\varphi\vert_N)}) \text{ as } r\to \infty.
	\end{align*}
	Define the set
	\[ F_N^a(r):=\{[i^a(x)]_{\varphi}\: \vert\: x\in B_N^{T}(r)\}. \]
	By Lemma \ref{lem:inclusion balls} we know that there exists some $r_0^a\in \R_{\geq 0}$ and $c_1^a,c_2^a\in \R_{>0}$ such that for all $r\geq r_0^a$ it holds that
	\[ i^a(B_N^{T}(c_1^ar))\subset B_G^S(r)\cap i^a(N)\subset i^a(B_N^{T}(c_2^ar)). \]
	In particular, we have for all $r\geq r_0^a$ that
	\[ F_N^a(c_1^ar) \subset F_G^a(r) \subset F_N^a(c_2^ar). \]
	Now we argue that $\#V^a(r)/\# E^a \leq \# F_N^a(r) \leq \# V^a(r)$ (for all $r\in \R_{\geq 0}$) and thus for all $r\geq r_0^a$ we have that
	\begin{equation}\label{eq:proof RG inequalities}
		\frac{\#V^a(c_1^ar)}{\# E^a} \leq \# F_N^a(c_1^ar) \leq \# F_G^a(r) \leq \# F_N^a(c_2^ar) \leq \# V^a(c_2^ar).
	\end{equation}
	For this, fix any $r\in \R_{\geq 0}$ and representatives $\tilde{R}\subset i^a(B_N^{T}(r))$ of different twisted conjugacy classes from $F_N^a(r)$. Define maps
	\[ \pi_1:E^a\times \tilde{R}\to V^a(r) \text{ with }\]
    \[\pi_1(c,i^a(x)):= \text{representative in } R \text{ of } \tau_c(x)-(\tau_a\circ\varphi^a)(c)+\Image(\Id_N-\tau_a\circ\varphi\vert_N) \]
	and
	\[ \pi_2:V^a(r)\to \tilde{R} \text{ with } \pi_2(x):= \text{representative in } \tilde{R} \text{ of } [i^a(x)]_{\varphi}.  \]
        Note that $\pi_1$ and $\pi_2$ are well-defined. Indeed, take $(c,i^a(x))\in E^a\times \tilde{R}$ and denote with $y\in R$ the unique representative of $\tau_c(x)-(\tau_a\circ\varphi^a)(c)+\Image(\Id_N-\tau_a\circ\varphi\vert_N)$. Hence, by Lemma \ref{lem:virt ab finite union} we obtain that $i^a(y)\sim_{\varphi} i^a(x)$ and thus $x\in \left(\bigcup_{d\in E^a} (\tau_d(y)-(\tau_a\circ\varphi^a)(d)+\Image(\Id_N-\tau_a\circ\varphi\vert_N)) \right)\cap B_N^{T}(r)$. In particular, we have that $y\in V^a(r)$. For $\pi_2$, take $x\in V^a(r)$ and take some $y\in \left(\bigcup_{c\in E^a} (\tau_c(x)-(\tau_a\circ\varphi^a)(c)+\Image(\Id_N-\tau_a\circ\varphi\vert_N)) \right)\cap B_N^{T}(r)$. By Lemma \ref{lem:virt ab finite union} we get that  $i^a(x)\sim_{\varphi} i^a(y)$ with $y\in B_N^{T}(r)$ and thus $[i^a(x)]_{\varphi}\in F_N^a(r)$. So we can indeed consider the unique representative of $[i^a(x)]_{\varphi}$ in $\tilde{R}$.
	
	We argue that $\pi_1$ and $\pi_2$ are surjective. For this, take an arbitrary $x\in V^a(r)$ and fix some $y\in \left(\bigcup_{c\in E^a} (\tau_c(x)-(\tau_a\circ\varphi^a)(c)+\Image(\Id_N-\tau_a\circ\varphi\vert_N)) \right)\cap B_N^{T}(r)$. By Lemma \ref{lem:virt ab finite union} we get that  $i^a(x)\sim_{\varphi} i^a(y)$ with $y\in B_N^{T}(r)$ and thus $[i^a(x)]_{\varphi}\in F_N^a(r)$. Take $i^a(z)\in \tilde{R}$ the unique representative of $[i^a(x)]_{\varphi}$ in $\tilde{R}$. Thus by Lemma \ref{lem:virt ab finite union} there exists some $c\in E^a$ such that $x\in \tau_c(z)-(\tau_a\circ\varphi^a)(c)+\Image(\Id_N-\tau_a\circ\varphi\vert_N)$. Since $x\in V^a(r)\subset R$, it holds that $x\in R$ is the unique representative of $\tau_c(z)-(\tau_a\circ\varphi^a)(c)+\Image(\Id_N-\tau_a\circ\varphi\vert_N)$ in $R$. Hence, we find that $\pi_1(c,i^a(z))=x$ and thus $\pi_1$ is surjective. For $\pi_2$, take any $i^a(x)\in \tilde{R}$ arbitrary and denote with $y\in R$ the unique representative of $x+\Image(\Id_N-\tau_a\circ\varphi\vert_N)$. By Remark \ref{rem:virt ab finite union} we obtain that $x\in \left(\bigcup_{c\in E^a} (\tau_c(y)-(\tau_a\circ\varphi^a)(c)+\Image(\Id_N-\tau_a\circ\varphi\vert_N)) \right)\cap B_N^{T}(r)$ and thus $y\in V^a(r)$. Moreover, by Lemma \ref{lem:virt ab finite union} we get that $i^a(x)\sim_{\varphi} i^a(y)$ and thus $i^a(x)\in \tilde{R}$ is the unique representative of $[i^a(y)]_{\varphi}$ in $\tilde{R}$. Thus, it follows that $\pi_2(y)=i^a(x)$. Hence, also $\pi_2$ is surjective.
	
	Since $\pi_1$ and $\pi_2$ are surjective (for any $r\in \R_{\geq 0}$), we obtain that $\#V^a(r)/\# E^a \leq \# F_N^a(r) \leq \# V^a(r)$ for all $r\in \R_{\geq 0}$. In particular, we can conclude the inequalities in Equation (\ref{eq:proof RG inequalities}) for all $r\geq r_0^a$. Using that $\# V^a(r)=\Theta(r^{\rank(N)-\rank(\Id_N-\tau_a\circ\varphi\vert_N)})$ as $r\to \infty$, it follows that
	\[ \# F_G^a(r)=\Theta\left(r^{\rank(N)-\rank(\Id_N-\tau_a\circ\varphi\vert_N)}\right) \text{ as } r\to \infty. \]
	Since $a\in A$ was taken arbitrary, this holds for all $a\in A$. Using Equation (\ref{eq:proof inequalities different layers}) and the definition of $\Theta$ (see Definition \ref{def:big Theta}) it now easily follows that
	\[ f_R^S(r)=\# F(r)=\Theta\left(r^{\rank(N)-\min\limits_{a\in A}\rank(\Id_N-\tau_a\circ\varphi\vert_N)}\right) \text{ as } r\to \infty. \]
	Since $N\subset_{\text{fin}} G$ and $\beta_N^{T}(r)=\Theta(r^{\rank(N)})$ as $r\to\infty$ (see for example \cite[Example 6.1.2]{loh17}), it follows (by for example \cite[Corollary 5.4.5, Proposition 6.2.4]{loh17}) that $\beta_G^S(r)=\Theta(r^{\rank(N)})$ as $r\to\infty$. In particular, by Corollary \ref{cor:twisted conjugacy growth independent generating set} it follows that
	\[ \beta_G^S(r)=\Theta\left(r^{\rank(N)}\right) \quad \text{and}\quad f_R^S(r)=\Theta\left(r^{\rank(N)-\min\limits_{a\in A}\rank(\Id_N-\tau_a\circ\varphi\vert_N)}\right) \text{ as } r\to \infty \]
	for all finite generating sets $S$ of $G$ and the twisted conjugacy growth of $G$ equals
	\[ 1-\frac{\min\limits_{a\in A}\rank(\Id_N-\tau_a\circ\varphi\vert_N)}{\rank(N)} \]
	for all finite generating sets $S$ of $G$.
\end{proof}


\section{Twisted conjugacy class growth}\label{sec:tc class growth}
Instead of studying the number of twisted conjugacy classes intersecting the balls $B_G^S(r)$, one could also be interested in the growth of one particular twisted conjugacy class $[g]_{\varphi}$. More precisely, for any subset $U\subset G$ we can define the \textit{subset growth} of $U$ in $G$ (with respect to $S$) as
\[ \beta_{U \subset G}^S(r):=\# B_{U\subset G}^S(r):=\# \{u\in U\:\vert \: \norm{u}_S\leq r\}. \]
For any $\varphi\in \End(G)$ and any $g\in G$, the \textit{twisted conjugacy class growth} of the twisted conjugacy class $[g]_{\varphi}$ equals $\beta_{[g]_{\varphi}\subset G}^S$.

Dermenjian and Evetts proved (see \cite{de23}) that any conjugacy class of a finitely generated virtually abelian group grows polynomially. Using Lemma \ref{lem:virt ab finite union}, we can prove the analogous result for all twisted conjugacy classes and determine the degree of this polynomial.
\begin{theorem}\label{thm:twisted conjugacy class growth}
    Let $S$ be any finite generating set of a finitely generated virtually abelian group $G$ and $N\lhd_{\text{fin}}G$ a torsion-free abelian subgroup of $G$. Let $\varphi\in \End(G)$ with $\varphi(N)\subset N$, then any twisted conjugacy class grows polynomially. More precisely, for any $g_0\in G$ it holds that
    \[\beta_{[g_0]_{\varphi}\subset G}^S(r)=\Theta(r^k) \text{ as }r\to \infty\quad \text{where}\]
    \[k=\max\{\rank(\Id_N-\tau_g\circ\varphi\vert_N)\:\vert\: gN\in G/N \text{ and } [g_0]_{\varphi}\cap p^{-1}(gN)\neq \emptyset\}.\]
\end{theorem}
\begin{proof}
    As before, fix a finite set $A\subset G$, including $1_G$, of representatives of $G/N$. Fix some $g_0\in G$ and define $A_0\subset A$ by
    \[A_0:=\{a\in A\:\vert\: [g_0]_{\varphi}\cap p^{-1}(aN)\neq \emptyset\}\]
    which contains the representatives of all cosets that intersect the twisted conjugacy class of $g_0$. For any $a\in A_0$ we choose some $x_a\in N$ such that $x_a a\sim_{\varphi} g_0$. Hence, it holds that
    \begin{align*}
        [g_0]_{\varphi}&=\bigcup_{a\in A_0}[g_0]_{\varphi}\cap p^{-1}(aN)\\
        &=\bigcup_{a\in A_0}i^a\left( \bigcup_{c\in E^a} (\tau_c(x_a)-(\tau_a\circ\varphi^a)(c)+\Image(\Id_N-\tau_a\circ\varphi\vert_N))\right)
    \end{align*}
    where in the last step we used Lemma \ref{lem:virt ab finite union}. Denote for simplicity (for any $a\in A_0$)
    \[V^a:=\bigcup_{c\in E^a} (\tau_c(x_a)-(\tau_a\circ\varphi^a)(c)+\Image(\Id_N-\tau_a\circ\varphi\vert_N)).\]
    Since $A_0$ is finite, it follows by Lemma 2.5 in \cite{de23} that $\beta_{[g_0]_{\varphi}\subset G}^S\sim \max\limits_{a\in A_0} \beta_{i^a(V^a)\subset G}^S$. Recall that $i^a:N\to G:x\mapsto xa$ is a quasi-isometry (see Lemma \ref{lem:inclusion balls}). Hence, Proposition 2.6 in \cite{de23} yields for any $a\in A$ that $\beta_{i^a(V^a)\subset G}^S\sim \beta_{V^a\subset \Z^n}^T$ (where $T$ is any finite generating set of $\Z^n$). Using Lemma 2.4 in \cite{de23} we obtain for any $a\in A_0$ and $c\in E^a$ that
    \begin{align*}
        \beta_{\tau_c(x_a)-(\tau_a\circ\varphi^a)(c)+\Image(\Id_N-\tau_a\circ\varphi\vert_N)\subset \Z^n}^T&\sim \beta_{\Image(\Id_N-\tau_a\circ\varphi\vert_N)\subset \Z^n}^T\\
        &\sim (r\mapsto r^{\rank(\Id_N-\tau_a\circ\varphi\vert_N)})
    \end{align*}
    where we used that $\Image(\Id_N-\tau_a\circ\varphi\vert_N)\cong \Z^{\rank(\Id_N-\tau_a\circ\varphi\vert_N)}$ grows polynomially of degree its rank. Applying Lemma 2.5 from \cite{de23} again and combining all the rest, we can conclude that
    \begin{align*}
        \beta_{[g_0]_{\varphi}\subset G}^S(r)&\sim \max\limits_{a\in A_0} (r\mapsto r^{\rank(\Id_N-\tau_a\circ\varphi\vert_N)})\\
        &\sim (r\mapsto r^{\max\limits_{a\in A_0}\rank(\Id_N-\tau_a\circ\varphi\vert_N)}).
    \end{align*}
    Lemma \ref{lem:asymptotics growth independent} now gives us the desired result.
\end{proof}


\section{Twisted conjugacy quotient growth}\label{sec:quotient growth}
Let $G$ be a finitely generated virtually abelian group with $\Z^n\lhd_{\text{fin}} G$. Note that in this section we use $\Z^n$ instead of $N$. Fix an endomorphism $\varphi\in \End(G)$ with $\varphi(\Z^n)\subset \Z^n$ and denote for any $k\in \N_0$ with $\olvarphi_k$ the induced morphism on $G/(k\Z)^n$. Since $G/(k\Z)^n$ is finite, it follows that $R(\olvarphi_k)<\infty$. We define the \textit{twisted conjugacy quotient growth function $f_Q$} as
\[f_Q:\N_0\to \N_0:k\mapsto R(\olvarphi_k).\]
We prove a slight modification of Lemma \ref{lem:virt ab finite union} where we use the set-up from Notation \ref{not:set-up and definition Ea}.
\begin{lemma}\label{lem:virt ab finite union quotient}
    It holds for any $a\in A$ and for any $x,y\in \Z^n$ that
    \[i^a(x)(k\Z)^n\sim_{\olvarphi_k} i^a(y)(k\Z)^n\quad \Longleftrightarrow\quad x\in \bigcup_{c\in E^a} (\tau_c(y)-(\tau_a\circ\varphi^a)(c)+(\Image(\Id_{\Z^n}-\tau_a\circ\varphi\vert_{\Z^n})+(k\Z)^n)).\]
\end{lemma}

\begin{proof}
    Fix any $a\in A$ and $x,y\in \Z^n$. Hence, it follows that
    \begin{align*}
		i^a(x)(k\Z)^n\sim_{\olvarphi_k} i^a(y)(k\Z)^n&\Longleftrightarrow \exists z\in \Z^n,\: c\in A:\: i^a(x)\in i^c(z)i^a(y)\varphi(i^c(z))^{-1}(k\Z)^n\\
		&\Longleftrightarrow \exists z\in \Z^n,\: c\in A:\: i^a(x)\in i^c(z)i^a(y)\varphi(c)^{-1}\varphi(z)^{-1}(k\Z)^n.
	\end{align*}
        Since $\varphi(\Z^n)\subset \Z^n$, by projecting onto $G/\Z^n$ we obtain that $a\Z^n=ca\varphi(c)^{-1}\Z^n$ and thus $c\in E^a$. Using the definition of $\varphi^a$ (see Equation (\ref{eq:definition Ea})) we get that
	\begin{align*}
		i^a(x)(k\Z)^n\sim_{\olvarphi_k} i^a(y)(k\Z)^n&\Longleftrightarrow \exists z\in \Z^n,\: c\in E^a:\: i^a(x)\in i^c(z)i^a(y)(\varphi^a(c)a^{-1}ca)^{-1}\varphi(z)^{-1}(k\Z)^n\\							  
		&\Longleftrightarrow \exists z\in \Z^n,\: c\in E^a:\: x\in z\tau_c(y)(\tau_a\circ\varphi^a)(c)^{-1}(\tau_a\circ\varphi)(z)^{-1}(k\Z)^n\\
		&\Longleftrightarrow \exists z\in \Z^n,\: c\in E^a:\: x\in \tau_c(y)-(\tau_a\circ\varphi^a)(c)+(\Id_{\Z^n}-\tau_a\circ\varphi\vert_{\Z^n})(z)+(k\Z)^n
	\end{align*}
	where in the last step we used that $\varphi^a(E^a)\subset \Z^n$, that $\Z^n$ is abelian and we switched to additive notation for convenience. Thus it follows that for all $x,y\in \Z^n$
		\[i^a(x)(k\Z)^n\sim_{\olvarphi_k} i^a(y)(k\Z)^n\quad \Longleftrightarrow\quad x\in \bigcup_{c\in E^a} (\tau_c(y)-(\tau_a\circ\varphi^a)(c)+(\Image(\Id_{\Z^n}-\tau_a\circ\varphi\vert_{\Z^n})+(k\Z)^n)).\]
\end{proof}
In order to determine the Reidemeister number of $\olvarphi_k$, we need the number of cosets of $\Image(B)+(k\Z)^n$ in $\Z^n$ for any morphism $B\in \End(\Z^n)$. For any subgroup $H\subset \Z^n$ let us denote with $\sqrt{H}:=\{z\in \Z^n\:\vert\: \exists m\in \N_0:\: mz\in H\}$ the isolator subgroup of $H$. In particular, it holds that $\sqrt{H}/H$ is finite. The next lemma is needed to determine $[\Z^n:H+(k\Z)^n]$.
\begin{lemma}\label{lem:order v_r}
    Let $0 \neq H\subset \Z^n$ and $k,l\in \N_0$. Denote $r:=\rank(H)$ and fix any $\Z$-basis $\{v_1,\dots,v_r\}$ of $H$. If $lv_r\in \langle v_1,\dots,v_{r-1}\rangle + (k\Z)^n$ (or $lv_1\in (k\Z)^n$ if $r=1$), then $k\leq \left\vert\frac{\sqrt{H}}{H}\right\vert l$.
\end{lemma}
\begin{proof}
    For $r=1$ the proof is completely analogous and thus we assume that $r\geq 2$. 
    Denote
    \[lv_r=\sum_{i=1}^{r-1} \lambda_i v_i +kz\]
    where $z,\lambda_i\in \Z$. In particular, we obtain that $kz\in H$ and thus $z\in \sqrt{H}$. Denote with $d$ the order of $z+H$ in $\sqrt{H}/H$. Thus $dz\in H$ and hence there are unique $\mu_i\in \Z$ such that
    \[dz=\sum_{i=1}^{r} \mu_i v_i.\]
    Since $kz\in H$, it follows that $d\mid k$. Hence, we obtain that
    \[\sum_{i=1}^{r} \frac{k}{d}\mu_i v_i=\frac{k}{d}dz=kz=lv_r-\sum_{i=1}^{r-1} \lambda_i v_i.\]
    Using that $\{v_1,\dots,v_r\}$ is a $\Z$-basis of $H$ we get that $\frac{k}{d}\mu_r=l$ and thus $k\leq dl$. Since $d\leq \left\vert\frac{\sqrt{H}}{H}\right\vert$ the result follows.
\end{proof}
\begin{theorem}\label{thm:index subgroup Z^n}
    Let $H\subset \Z^n$, then
    \[[\Z^n:H+(k\Z)^n]=\Theta(k^{n-\rank(H)}) \text{ as } k\to \infty\]
\end{theorem}
\begin{proof}
    Note that
    \[[\Z^n:H+(k\Z)^n]=\frac{[\Z^n:(k\Z)^n]}{[ H+(k\Z)^n:(k\Z)^n]}=\frac{k^n}{[H+(k\Z)^n:(k\Z)^n]}\]
    and thus it suffices to argue that $[H+(k\Z)^n:(k\Z)^n]=\Theta(k^{\rank(H)})$ as $k\to \infty$. We argue this by using induction on the rank of $H$. If $\rank(H)=0$, then the result holds trivially. Let us assume that $\rank(H)\in \N_0$ and the result holds for any subgroup of rank less than $\rank(H)$. Denote $r:=\rank(H)$ and fix any $\Z$-basis $\{v_1,\dots,v_r\}$ of $H$. To simplify notation, $\langle v_1,\dots,v_{r-1}\rangle:=0$ if $r=1$. Since the rank of the subgroup $\langle v_1,\dots,v_{r-1}\rangle$ equals $r-1$, it follows by the induction hypothesis that
    \[[\langle v_1,\dots,v_{r-1}\rangle+(k\Z)^n:(k\Z)^n]=\Theta(k^{r-1})\text{ as } k\to \infty.\]
    Fix any $k>\left\vert\frac{\sqrt{H}}{H}\right\vert$. Thus by Lemma \ref{lem:order v_r} it holds that $v_r\not\in \langle v_1,\dots,v_{r-1}\rangle+(k\Z)^n$. Hence, $v_r+\langle v_1,\dots,v_{r-1}\rangle+(k\Z)^n$ is a generator of the finite cyclic group $(H+(k\Z)^n)/(\langle v_1,\dots,v_{r-1}\rangle+(k\Z)^n)$ and the size of this group thus equals the order of $v_r+\langle v_1,\dots,v_{r-1}\rangle+(k\Z)^n$. Using Lemma \ref{lem:order v_r} we obtain that
    \[\frac{k}{\left\vert\frac{\sqrt{H}}{H}\right\vert} \leq [H+(k\Z)^n:\langle v_1,\dots,v_{r-1}\rangle+(k\Z)^n] \leq k.\]
    Since this holds for all $k>\left\vert\frac{\sqrt{H}}{H}\right\vert$, it follows that
    \[[H+(k\Z)^n:\langle v_1,\dots,v_{r-1}\rangle+(k\Z)^n]=\Theta(k) \text{ as } k\to \infty\]
    and thus
    \begin{align*}
        [H+(k\Z)^n:(k\Z)^n]&=[H+(k\Z)^n:\langle v_1,\dots,v_{r-1}\rangle+(k\Z)^n] \cdot [\langle v_1,\dots,v_{r-1}\rangle+(k\Z)^n:(k\Z)^n]\\
        &=\Theta(k^{\rank(H)}) \text{ as } k\to \infty
    \end{align*}
    which concludes the proof by induction.
\end{proof}

\begin{theorem}\label{thm:quotient growth}
    If $G$ is a finitely generated virtually abelian group with $\Z^n\lhd_{\text{fin}} G$ and $\varphi\in \End(G)$ with $\varphi(\Z^n)\subset \Z^n$, then
    \[f_Q(k)=\Theta\left(k^{n-\min\limits_{g\Z^n\in G/\Z^n}\rank(\Id_{\Z^n}-\tau_g\circ\varphi\vert_{\Z^n})}\right) \text{ as } k\to \infty.\]
\end{theorem}
\begin{proof}
    Recall the set-up from Notation \ref{not:set-up and definition Ea}. For any $a\in A$ we define the sets
    \[ \olF(k):=\{[g(k\Z)^n]_{\olvarphi_k}\: \vert\: g\in G\} \quad \text{and} \quad \olFa(k):=\{[i^a(x)(k\Z)^n]_{\olvarphi_k}\: \vert\: x\in \Z^n\}. \]
    Similarly as in the proof of Theorem \ref{thm:twisted conjugacy growth virtually abelian} it follows that
	\[ \olF(k)= \bigcup_{a\in A} \olFa(k) \]
	and thus
	\begin{equation}\label{eq:proof quotient inequalities different layers}
		\max_{a\in A} \# \olFa(k)\leq \# \olF(k)=R(\olvarphi_k)=f_Q(k) \leq \sum_{a\in A} \# \olFa(k).
	\end{equation}
    Thus again it suffices to argue for all $a\in A$ that $\# \olFa(k)=\Theta\left(k^{n-\rank(\Id_{\Z^n}-\tau_a\circ\varphi\vert_{\Z^n})}\right)$ as $k\to \infty$. So fix any $a\in A$ and denote $H:=\Image(\Id_{\Z^n}-\tau_a\circ\varphi\vert_{\Z^n})$. We claim that
    \begin{equation}{\label{eq:proof quotient inequality finite union}}
        \frac{[\Z^n:H+(k\Z)^n]}{\# E^a}\leq \# \olFa(k) \leq [\Z^n:H+(k\Z)^n].
    \end{equation}
    To prove this claim, we fix a set $R\subset \Z^n$ of representatives of $\Z^n/(H+(k\Z)^n)$ and a set $\tilde{R}\subset i^a(\Z^n)$ such that $\{i^a(x)(k\Z)^n\:\vert\: i^a(x)\in \tilde{R}\}$ contains representatives of the $\olvarphi_k$-conjugacy classes in $\olFa(k)$. In particular, we have that $\# R=[\Z^n:H+(k\Z)^n]$ and $\# \tilde{R}=\# \olFa(k)$. Define the maps $\pi_1$ and $\pi_2$ by setting
    \[ \pi_1:E^a\times \tilde{R}\to R \text{ with }\]
    \[\pi_1(c,i^a(x)):= \text{representative in } R \text{ of } \tau_c(x)-(\tau_a\circ\varphi^a)(c)+(H+(k\Z)^n) \]
    and
    \[ \pi_2:R\to \tilde{R} \text{ with } \pi_2(x)\in \tilde{R} \text{ such that } \pi_2(x)(k\Z)^n\sim_{\olvarphi_k}i^a(x)(k\Z)^n. \]
    The maps $\pi_1$ and $\pi_2$ are clearly well-defined. We argue that they are both surjective. First, we fix some $r\in R\subset \Z^n$. Take $i^a(x)\in \tilde{R}$ such that $i^a(x)(k\Z)^n\sim_{\olvarphi_k} i^a(r)(k\Z)^n$. Hence, by Lemma \ref{lem:virt ab finite union quotient} we can choose some $c\in E^a$ such that $r\in \tau_c(x)-(\tau_a\circ\varphi^a)(c)+(H+(k\Z)^n)$. Hence, it follows that $\pi_1(c,i^a(x))=r$ and thus $\pi_1$ is surjective. Next, we take any $i^a(y)\in \tilde{R}$. Fix a representative $x\in R\subset \Z^n$ of $y+(H+(k\Z)^n)$. Since $1_G\in E^a$ and $(\tau_a\circ \varphi^a)(1_G)=1_G$, it follows by Lemma \ref{lem:virt ab finite union quotient} that $i^a(x)(k\Z)^n\sim_{\olvarphi_k} i^a(y)(k\Z)^n$ and thus $\pi_2(x)=i^a(y)$. Hence, $\pi_1$ and $\pi_2$ are both surjective and thus Equation (\ref{eq:proof quotient inequality finite union}) follows.

    Combining Theorem \ref{thm:index subgroup Z^n} and Equation (\ref{eq:proof quotient inequality finite union}) yields that
    \[\# \olFa(k)=\Theta\left(k^{n-\rank(\Id_{\Z^n}-\tau_a\circ\varphi\vert_{\Z^n})}\right) \text{ as } k\to \infty.\]
    Since this holds for all $a\in A$, we can conclude by using Equation (\ref{eq:proof quotient inequalities different layers}) that
    \[f_Q(k)=\Theta\left(k^{n-\min\limits_{a\in A}\rank(\Id_{\Z^n}-\tau_a\circ\varphi\vert_{\Z^n})}\right) \text{ as } k\to \infty.\]
\end{proof}

\begin{corollary}\label{cor:tc and quotient growth equivalent}
    If $G$ is a virtually abelian group that is finitely generated by $S$ and $N\lhd_{\text{fin}}G$ is a torsion-free abelian subgroup of $G$. If $\varphi\in \End(G)$ with $\varphi(N)\subset N$, then the twisted conjugacy growth $f_R^S$ and the quotient growth $f_Q$ of $G$ are quasi-equivalent, i.e. $f_R^S\sim f_Q$.\\
    Moreover,
	\[ f_R^S(k)=\Theta(f_Q(k))=\Theta\left(k^{\rank(N)-\min\limits_{gN\in G/N}\rank(\Id_N-\tau_g\circ\varphi\vert_N)}\right)\text{ as } k\to \infty. \]
\end{corollary}


\section{Examples}\label{sec:examples}
In this section, we work out some concrete examples regarding the different growth functions. For this, we describe a specific family of virtually abelian groups for which Lemma \ref{lem:virt ab finite union} can be strengthened.
\begin{lemma}\label{lem:examples virt ab finite union}
    Let $G$ be a finitely generated virtually abelian group with $N\lhd_{\text{fin}} G$ a torsion-free abelian subgroup and $\varphi\in \End(G)$ with $\varphi(N)\subset N$. If under the set-up from Notation \ref{not:set-up and definition Ea}, also $G/N$ is abelian and the induced morphism $\olvarphi\in \End(G/N)$ is the identity on $G/N$, then for any $a,b\in A$ and for any $x,y\in N$ it holds that
    \[i^a(x)\sim_{\varphi} i^b(y)\quad \Longleftrightarrow\quad \begin{cases}
        a=b\\
        x\in \bigcup\limits_{c\in A} (\tau_c(y)-(\tau_a\circ\varphi^a)(c)+\Image(\Id_N-\tau_a\circ\varphi\vert_N))
    \end{cases}.\]
\end{lemma}
\begin{proof}
    Note that $i^a(x)\sim_{\varphi} i^b(y)$ if and only if there exists some $c\in A$ and $z\in N$ such that
    \[i^a(x)=i^c(z)i^b(y)\varphi(i^c(z))^{-1}.\]
    Hence, by projecting this equation onto $G/N$, we obtain that $aN=cb\varphi(c)^{-1}N$. Since $G/N$ is abelian and $\varphi(c)N=cN$, we obtain that $aN=bN$ and thus $a=b$ by definition of $A$. Moreover, it holds that $E^a=\{c\in A\:\vert \: aN=ca\varphi(c)^{-1}N\}=A$. Using Lemma \ref{lem:virt ab finite union}, we can now conclude the result.
\end{proof}
From now on, we assume that $G/N$ is abelian and that $\olvarphi=\Id_{G/N}$. In particular, for an arbitrary element $i^a(x)\in G$ (with $a\in A$ and $x\in N$) the $\varphi$-twisted conjugacy class of $i^a(x)$ will be contained in $i^a(\Z^n)$. More precisely, it holds that
\[[i^a(x)]_{\varphi}=i^a\left(\bigcup\limits_{c\in A} (\tau_c(x)-(\tau_a\circ\varphi^a)(c)+\Image(\Id_N-\tau_a\circ\varphi\vert_N))\right).\]
In other words, the projection of any twisted conjugacy class onto $G/N$ is a singleton. Hence, by Theorem \ref{thm:twisted conjugacy class growth} we obtain the following result.
\begin{theorem}\label{thm:examples twisted conjugacy class growth}
    Let $S$ be any finite generating set of a finitely generated virtually abelian group $G$ and $N\lhd_{\text{fin}}G$ a torsion-free abelian subgroup of $G$. Let $\varphi\in \End(G)$ with $\varphi(N)\subset N$ and $\olvarphi=\Id_{G/N}$. If $G/N$ is abelian, then for any $g_0\in G$ it holds that
    \[\beta_{[g_0]_{\varphi}\subset G}^S(r)=\Theta(r^{\rank(\Id_N-\tau_{g_0}\circ\varphi\vert_N)}) \text{ as }r\to \infty.\]
\end{theorem}
\begin{proof}
    By Lemma \ref{lem:examples virt ab finite union} it follows that
    \[\{gN\in G/N\:\vert \: [g_0]_{\varphi}\cap p^{-1}(gN)\neq \emptyset\}=\{g_0N\}.\]
    Hence, we can conclude the result by applying Theorem \ref{thm:twisted conjugacy class growth}.
\end{proof}

In the next example, we illustrate our results for a concrete group and some morphisms.
\begin{example}
    Consider the semi-direct product
    \[G:=\Z^2\semi\Z/2\Z\]
    where the non-trivial element $t$ of $\Z/2\Z$ acts as $-\Id_{\Z^2}$ on $\Z^2$. Hence, the set $A:=\{1_G,t\}$ is a set of representatives of the cosets in $G/\Z^2\cong \Z/2\Z$. Fix any finite generating set $S$ of $G$ and denote with $\varphi\in \End(G)$ the morphism for which $\varphi\vert_{\Z^2}=-\Id_{\Z^2}$ and $\varphi(t)=t$. Note that $\tau_t\circ \varphi\vert_{\Z^2}=\Id_{\Z^2}$ and thus
    \[\rank(\Id_{\Z^2}-\tau_g\circ \varphi\vert_{\Z^2})=\begin{cases}
        2 &\text{if } g\in \Z^2\\
        0 &\text{if } g\in t\Z^2
    \end{cases}.\]
    Thus by Theorem \ref{thm:twisted conjugacy growth virtually abelian}, Theorem \ref{thm:quotient growth} and Theorem \ref{thm:examples twisted conjugacy class growth} it follows that as $r\to \infty$
    \[\beta_{[g_0]_{\varphi}\subset G}^S(r)=\begin{cases}
            \Theta(r^2) &\text{if } g_0\in \Z^2\\
            \Theta(1) &\text{if } g_0\in t\Z^2
        \end{cases} \text{, }\quad f_R^S(r)=\Theta(r^2)\quad \text{ and }\quad f_Q(r)=\Theta(r^2).\]
    Note that the asymptotics of $\beta_{[g_0]_{\varphi}\subset G}^S$ only depends on the coset of $\Z^2$ that contains $g_0$ and thus is the same for any two elements belonging to the same coset. Moreover, the degree of both $f_R^S$ and $f_Q$ equals two (being the rank of $\Z^2$) minus the minimum of the degrees occurring in the description of $\beta_{[g_0]_{\varphi}\subset G}^S$. These observations hold always if $G/N$ is abelian and the induced morphism $\olvarphi=\Id_{G/N}$.\\
    Next, one can check that if $\varphi=\Id_G$, then still $f_R^S(r)=\Theta(f_Q(r))=\Theta(r^2)$ as $r\to \infty$, but
    \[\beta_{[g_0]_{\varphi}\subset G}^S(r)=\begin{cases}
            \Theta(1) &\text{if } g_0\in \Z^2\\
            \Theta(r^2) &\text{if } g_0\in t\Z^2
        \end{cases} \text{ as } r\to\infty.\]
    In particular, the twisted conjugacy class of the identity element does not necessarily grow faster or slower than any other class. It is also possible that the twisted conjugacy class growth is (asymptotically) independent of the considered twisted conjugacy class. One can check that for example for the morphism $\varphi\in \End(G)$ which is defined by
    \[\varphi\vert_{\Z^2}=\begin{pmatrix}
        -1 & 0 \\
        0 & 1
    \end{pmatrix}\quad \text{and}\quad \varphi(t)=t\]
    it holds for any $g_0\in G$ that
    \[f_R^S(r)=\Theta(\beta_{[g_0]_{\varphi}\subset G}^S(r))=\Theta(f_Q(r))=\Theta(r)\text{ as } r\to\infty.\]
\end{example}


\end{document}